\documentclass[a4paper,11pt,twoside]{amsart}

\usepackage[latin1]{inputenc}
\usepackage{amssymb,amsthm}
\usepackage[all]{xy}
\usepackage{amsmath}
\usepackage{indentfirst}
\usepackage{fancyhdr}
\usepackage{amsfonts}
\usepackage{textcomp}
\usepackage{xypic}
\usepackage{graphicx}
\usepackage{enumerate}
\usepackage{verbatim}
\usepackage{hyperref}
\usepackage{amsthm}
\usepackage{pifont}
\usepackage{eucal}
\usepackage{amssymb}
\usepackage{mathrsfs}

\newtheorem{teo}{Theorem}[section]
\newtheorem{pro}[teo]{Proposition}
\newtheorem{lem}[teo]{Lemma}
\newtheorem{co}[teo]{Corollary}

\newtheorem{con}[teo]{Conjecture}

\theoremstyle{definition}
\newtheorem{de}[teo]{Definition}

\newtheorem{ex}[teo]{Example}

\newtheorem{re}[teo]{Remark}

\newcommand{\Q}{\mathbb{Q}}
\newcommand{\C}{\mathbb{C}}
\newcommand{\R}{\mathbb{R}}

\newcommand{\N}{\mathbb{N}}
\newcommand{\OO}{\mathcal{O}}

\title{Divisorial models of normal varieties}
\author{Stefano Urbinati}
\date{}

\begin{document}

\keywords{Normal varieties, singularities of pairs, multiplier ideal sheaves}

\subjclass[2000]{14J17, 14F18, 14Q15}

\begin{abstract}
We prove that the canonical ring of a canonical variety in the sense of \cite{dFH} is finitely generated. We prove that canonical varieties are klt if and only if $\mathscr{R}_X(-K_X)$ is finitely generated. We introduce a notion of nefness for non-$\Q$-Gorenstein varieties and study some of its properties. We then focus on these properties for non-$\Q$-Gorenstein toric varieties.
\end{abstract}

\maketitle

\section{Introduction}\label{sec:intro}

In this paper we continue the investigation of singularities of non-$\Q$-Gorenstein varieties initiated in \cite{dFH} and \cite{Urb1}. In particular we focus on the study of canonical singularities and non-$\Q$-Gorenstein toric varieties.

In the context of the Minimal Model Program, there are some natural classes of singularities to be considered such as canonical and log terminal (cf. \cite{BCHM}).  Also, while running the program, we naturally encounter non-$\Q$-Gorenstein varieties, for which the canonical divisor is not $\Q$-Cartier, as target of small contractions when performing a flip. In 2009, de Fernex and Hacon in \cite{dFH} extended the notion of canonical and log terminal singularities to the case of non-$\Q$-Gorenstein varieties. It is natural to wonder which of the properties connected to theese type of singularities are preserved when extending to this more general context.

In the second section we will briefly recall the fundamental definitions and properties introduced in \cite{dFH}.

In the third section we show that if $X$ is canonical in the sense of \cite{dFH}, then the relative canonical ring $\mathscr{R}_X(K_X)$ is a finitely generated $\OO_X$-algebra (Theorem \ref{fgcr}). Thus, if $X$ is canonical, there exists a small proper birational morphism $\pi: X' \to X$ such that $K_{X'}$ is $\Q$-Cartier and $\pi$-ample. As a corollary we obtain that the canonical ring of any normal variety with canonical singularities (in the sense of \cite{dFH}) is finitely generated. 

We next turn our attention to log terminal singularities. Recall that in \cite{Urb1} we gave an example of canonical singularities that are not log terminal. In this paper we show that, if $X$ is canonical, then finite generation of the relative  anti-canonical ring $\mathscr{R}_X(-K_X)$ is equivalent to $X$ being log terminal (Proposition \ref{ltcan}).

In the fourth section we introduce a notion of nefness for Weil divisors (on non-$\Q$-factorial varieties). We call such divisors quasi-nef (q-nef) and we study their basic properties. We prove that if $X$ ia a normal variety with canonical singularities such that $K_X$ is q-nef,  then $X'= \text{Proj}_X(\mathscr{R}_X(K_X))$ is a minimal model.

In the last two sections, we focus our attention on toric varieties. We give a complete description of quasi-nef divisors on toric varieties and we notice that they correspond to divisors whose divisorial sheaf is globally generated.

In the last section we give a new natural definition of minimal log discrepancies (MLD) in the new setting and we prove that even in the toric case they do not satisfy the ACC conjecture. 

\bigskip

{\small\noindent {\bf Acknowledgments.} } The author would like to thank his supervisor C. D. Hacon for his support and the several helpful discussions and suggestions. The author is extremely grateful to the referee for the many suggestions, expecially for a shorter version of the proof of Proposition \ref{ltcan}.

\section{background}
We work over the complex numbers.

Recall the following definition of de Fernex and Hacon (cf \cite{dFH}):

\begin{de} Let $f: Y\to X$  a proper projective birational morphism of normal varieties. Given any Weil divisor $D$ on $X$, we define the pullback of $D$ on $Y$ as:
$$f^*(D) = \sum_{P\text{ prime on }Y} \lim_{k\to \infty} \frac{v_P(\OO_Y \cdot  \OO_X(- kD))}{k}\cdot P$$

\end{de}

Note that if $D$ is $\Q$-Cartier, then $f^*(D)$ coincides with the usual notion of pullback.

Using this definition, de Fernex and Hacon define canonical and log terminal singularities for non-$\Q$-Gorenstein varieties. As usual, this is done in terms of the relative canonical divisor $K_{Y/X}$, for $f:Y \to X$ a proper morphism. Note however that there are two different choices for the relative canonical divisor (which coincide in the $\Q$-Gorenstein setting):
$$K^-_{Y/X}:= K_Y -f^*(K_X) \quad \text{ and } \quad K^+_{Y/X}:= K_Y + f^*(-K_X).$$

We will not use $K^-_{Y/X}$ in this paper, but recall that it is the one used to define log terminal singularities and multiplier ideal sheaves.

\begin{de} $X$ is said to be \emph{canonical} if 
$$\text{ord}_F(K^+_{Y/X}) \geq 0$$
 for every exceptional prime divisor $F$ over $X$. 
\end{de}

\begin{de} $X$ is said to be \emph{log terminal} if and only if there is a an effective $\Q$-divisor $\Delta$ such that:
\begin{itemize}
\item $\Delta$ is a boundary ($K_X+ \Delta$ is $\Q$-Cartier) and
\item $(X,\Delta)$ is klt.
\end{itemize}
\end{de}

Let us also recall some standard definitions from the Minimal Model Program (for references, \cite{KM},\cite{BCHM}).
\begin{de} Let $X$ be a normal projective variety and let $D$ be a $\Q$-divisor on $X$. We define
$$
R_X(D):=\bigoplus_{m\geq1}H^0(\OO_X(\lfloor mD\rfloor))
$$
and
$$
\mathscr{R}_X(D):=\bigoplus_{m\geq1}\OO_X(\lfloor mD\rfloor).
$$
\end{de}

Where the first will be always considered as a $\mathbb{C}$-algebra and the second as an $\OO_X$-algebra.

One important property the we will often need in the background is the following result (\cite[6.2]{KM}).
\begin{lem}\label{lm:themostusefullemmaintheworld} Let $X$ be a normal algebraic variety and let $B$ be a Weil divisor on it. The following are equivalent:
\begin{enumerate}
\item $\mathscr{R}_X(B)$ is a finitely generated $\OO_X$-algebra;
\item there exists a small, projective birational morphism $f:Y\rightarrow X$ such that $Y$ is normal, and $\bar{B}:=f^{-1}_*B$ is $\Q$-Cartier and $f$-ample.
\end{enumerate}
Moreover, $f:Y\rightarrow X$ is unique with these properties, namely $Y\cong{\rm Proj}_X(\mathscr{R}_X(B))$, and, for all $m\geq0$, $f_*\OO_Y(m\bar{B})=\OO_X(mB)$.
\end{lem}

\section{Canonical singularities}
In this section we will show that if $X$ has canonical singularities, then its canonical ring is finitely generated.

de Fernex and Hacon gave the following characterization of canonical singularities:

\begin{pro} \cite[Proposition 8.2]{dFH} Let $X$ be a normal variety. Then $X$ is canonical if and only if for all sufficiently divisible $m\geq 1$, and for every resolution $f:Y \to X$, there is an inclusion 
$$\OO_X(mK_X)\cdot \OO_Y \subseteq \OO_Y (mK_Y )$$ as sub-$\OO_Y$-modules of $\mathcal{K}_Y$.

\end{pro}

\begin{lem} \cite[Lemma 2.14]{Urb1} \label{can1} Let $f:Y\to X$ be a proper birational morphism such that $Y$ is $\Q$-Gorenstein with canonical singularities. If $\OO_Y\cdot \OO_X(mK_X) \subseteq \OO_Y(m K_Y)$ for any sufficiently divisible $m\geq 1$, then $X$ is canonical.
\end{lem}

The following immediate corollary of Lemma \ref{can1} is very useful.

\begin{co} Let $f:Y\to X$ be a proper birational morphism such that $Y$ is $\Q$-Gorenstein and canonical. If ${\rm ord}_F(K^+_{Y/X})\geq 0$ for all divisors $F$ on $Y$, then $X$ is canonical.
\end{co}
\begin{proof} Set $f^{\natural}(D):= \text{div}(\OO_X(-D)\cdot \OO_Y)$. 

For all sufficiently divisible $m\geq 1$, $\mbox{ord}_F(K^+_{m, Y/X})\geq 0$ (i.e. $mK_Y \geq -f^{\natural}(-mK_X)$), so that:
$$\OO_Y\cdot \OO_X(mK_X) \hookrightarrow (\OO_Y\cdot \OO_X(mK_X))^{\vee \vee} = \OO_Y(-f^{\natural}(-mK_X)) \hookrightarrow \OO_Y(mK_Y).$$
\end{proof}

The first result that we will prove is that if $X$ is canonical, then $\mathscr{R}_X(K_X)$ is finitely generated over $X$. Note that this result is trivial for $\Q$-Gorenstein varieties.

\begin{teo}  \label{fgcr} If $X$ is canonical, then $\mathscr{R}_X(K_X )$ is finitely generated over $X$.
\end{teo}
\begin{proof} We may assume that $X$ is affine.
Let $\tilde X\to X$ be a resolution. By \cite{BCHM} $\mathscr{R}_X(K_{\tilde X})$ is finitely generated. Let $X^c={\rm Proj }_X(\mathscr{R}_X(K_{\tilde X}))$ and let $f:X^c\to X$ be the induced morphism, where $X^c$ is canonical. 
Since $X$ is canonical, for any $m>0$, there is an inclusion $\mathcal O _{X^c}\cdot \mathcal O _X(mK_X)\to \mathcal O _{X^c}(mK_{X^c})$. Pushing this forward we obtain inclusions
$$f_*(\mathcal O _{X^c}\cdot \mathcal O _X(mK_X))\subset f_*\mathcal O _{X^c}(mK_{X^c})\subset \mathcal O _X(mK_X).$$
Since the left and right hand sides have isomorphic global sections, then $H^0(f_*\mathcal O _{X^c}(mK_{X^c}))\cong H^0(\mathcal O _X(mK_X))$. Since $X$ is affine, $\OO_X(mK_X)$ is globally generated and hence $f_*\mathcal O _{X^c}(mK_{X^c})=\mathcal O _X(mK_X)$. But then $\mathscr{R}_X(K_X)\cong \mathscr{R}_X(K_{X^c}) :=\bigoplus f_*\OO_{X^c}(mK_{X^c})$ is finitely generated.
\end{proof}

\begin{re} Note that we have seen that 
$$\mathscr{R}_X(K_X) \cong \mathscr{R}_X(K_{X^c})  \cong \mathscr{R}_X(K_{\tilde{X}}) $$
hence $$X^c = \text{Proj}_X(\mathscr{R}_X(K_X))$$
and so $X^c \to X$ is a small morphism.
\end{re}

\begin{co}  If $X$ is canonical, then the canonical ring $\mathscr{R}_X(K_X )$ is finitely generated.
\end{co}

\begin{proof} Since $f: X^c \to X$ is small, it follows that $\mathscr{R}_X(K_X) \cong \mathscr{R}_X(K_{X^c})$. Since $X^c$ is canonical and $\Q$-Gorenstein if follows that $\mathscr{R}_X(K_{X^c})$ is finitely generated (cf. \cite{BCHM}).

\end{proof}

The next Proposition, strictly relates log terminal singularities with the finite generation of the canonical ring even in the non-$\Q$-Gorenstein case:

\begin{pro} \label{ltcan} Let $X$ be a normal variety with at most canonical singularities. $\mathscr{R}_X(-K_X)$ is a finitely generated $\OO_X$-algebra if and only if  $X$ is log terminal.
\end{pro}
\begin{proof} 
If $X$ is log terminal, then $\mathscr{R}_X(-K_X)$ is a finitely generated $\OO_X$-algebra by \cite[Theorem 92]{KollarEx}.

We now want to show that, if  $\mathscr{R}_X(-K_X)$ is finitely generated and $X$ has at most canonical singularities, than it is also log terminal. 

We will start constructing a $\mathbb{Q}$-Gorenstein model of $X$. Since $\mathscr{R}_X(-K_X)$ is finitely generated, by \cite[Proposition 6.2]{KM}, there exists a small map $\pi: X^-\to X$, such that $-K_{X^-}=\pi^{-1}_*(-K_X)$ is $\Q$-Cartier and $\pi$-ample. We will first show that $X$ having at most canonical singularities implies for the model $X^-$ to have at most canonical singularities in the usual sense.

For any $m$ sufficiently divisible, consider the natural map $\OO_{X^-}\cdot \OO_X(-mK_X) \to  \OO_{X^-}(-mK_{X^-}) $. Since $-K_{X^-}$ is $\pi$-ample, we can choose $A \subseteq X$ an ample divisor so that $\OO_X(-mK_X + A)$ and $\OO_X(-mK_{X^-} + \pi^*A)$ are both globally generated. The small map induces an isomorphism at the level of global sections 
$$H^0(X^-, \OO_{X^-}\cdot \OO_X(-mK_X +A)) \to  H^0(X^-, \OO_{X^-}(-mK_{X^-} + \pi^*A)) $$
and since the two sheaves are globally generated, this induces an isomorphism of sheaves $\OO_{X^-}\cdot \OO_X(-mK_X) \tilde{\rightarrow}  \OO_{X^-}(-mK_{X^-}) $.

Thus, considering $f: Y\to X$ and $g: Y \to X^-$, a common log resolution of both $X$ and $X^-$, we have
$$K_Y + \frac{1}{m}g^*(-mK_{X^-})=K_Y + \frac{1}{m}g^*(\pi^{\natural}(-mK_{X}))=K_Y + \frac{1}{m} f^{\natural}(-mK_X) \geq 0$$
the last inequality because $X$ has at most canonical singularities. Thus $X^-$ has at most canonical singularities. More, since $K_{X^-}$ is $\mathbb Q$-Cartier and canonical, $X^-$ is log terminal. \\
Choosing a general ample $\Q$-divisor $H^-\sim _{\mathbb Q, X}-K_{X^-}$, let $m\gg0$ and $G^- \in |mH^-|$ a general irreducible component. Picking $\Delta^- := \frac{G^-}{m}$ then $K_{X^-}+ \Delta^-\sim _{\mathbb Q, X}0$ is still log terminal and, because the sum is $\pi$-trivial, $\Delta^-$ will induce a boundary on $X$ so that $(X, \Delta=\pi_*\Delta^-)$ is a log terminal pair on $X$.
\end{proof}

\section{Quasi-nef divisors}

Given a divisor $D$ on a variety $X$, it is useful to know if the divisor is nef. In particular, varieties such that the canonical divisor $K_X$ is nef, are minimal models.\\
For arbitrary normal varieties, unfortunately, there is no good notion of nefness (this is a numerical property that is well defined if the variety is $\Q$-factorial). A possible definition is given in \cite{BdFF} using the notion of $b$-divisors. A comparison between the approach given in this paper and the one in \cite{BdFF} is given in detail in \cite{CU2}. 
 In particular, whenever looking for a minimal model in this case it is always necessary to either pass to a resolution of the singularities or to perturb the canonical divisor adding a boundary (an auxiliary divisor $\Delta$ such that $K_X + \Delta$ is $\Q$-Cartier). However both operations are not canonical and in either cases different choices lead us to different minimal models. What we would like to do in this section is to define a notion of a  minimal model for an arbitrary normal variety.

We will start defining a notion of nefness for a divisor that is not $\Q$-Cartier.

\begin{de} \label{defqn} Let $X$ be a projective normal variety. A divisor $D\subseteq X$ is \emph{quasi-nef} (\emph{q-nef}) if for every ample $\Q$-divisor $A \subseteq X$, $\OO_X(m(D+A))$ is generated by global sections for every $m >0$ sufficiently divisible.
\end{de}

\begin{re} Let $X$ be a normal $\Q$-factorial variety. A divisor $D\subseteq X$ is nef if and only if it is q-nef.
\end{re}

\begin{pro} \label{smallqn} Let $D$ be a divisor on a normal variety $X$. If $g: Y \to X$ is a small projective birational map such that $\bar{D}:=g^{-1}_*D$ is $\Q$-Cartier and $g$-ample, then $D$ is q-nef if and only if $\bar{D}$ is nef.
\end{pro}

\begin{proof}  The existance of the map $g$ is equivalent to the finite generation of $\mathscr{R}_X(D)$ as an $\OO_X$-algebra (\cite[6.2]{KM}).

Let us first assume that $D$ is q-nef. For every ample divisor $A \subseteq X$, by definition there exists a positive integer $m$ such that $\OO_X(m(D+A))$ is generated by global sections and $\OO_Y(m\bar{D})$ is relatively globally generated. In particular, since $g$ is small,  
$$\varphi: \OO_Y\cdot \OO_X(m(D+A)) \to \OO_Y( m(\bar{D}+ g^*A))$$
 induces an isomorphism at the level of global sections. Now $\bar{D}$ is $g$-ample, and there exists $k\gg0$ such that $\OO_Y(m(\bar{D} + g^*A))\otimes \OO_Y(kg^*A)$  is also generated by global sections, hence $\varphi$ must be surjective and hence an isomorphism. Since $\OO_Y \cdot \OO_X(m(D+A))$ is generated by global sections, so is $\OO_Y (m(\bar{D} + g^*A))$. This implies that $\bar{D} + g^*A$  is nef, and since nefness is a closed property, $\bar{D}$ is nef.

Let us now suppose that $\bar{D}$ is a nef divisor on $Y$. As in Proposition \ref{ltcan}, since $g$ is a small map we have an isomorphism of sheaves $g_*\OO_Y(m \bar{D})\cong \OO_X(mD)$. Even more, since $\bar{D}$ is $g$-ample, for any ample divisor $A$ on $X$, $\varepsilon \bar{D} +g^*A$ is ample for all $0< \varepsilon \ll 1$. If we assume that $\bar{D}$ is nef, then $$\bar{D} + g^*A= (1 -\varepsilon)\bar{D} + (\varepsilon \bar{D} + g^*A)$$

is ample, and $\OO_Y(m(\bar{D} + g^*A ))$ is thus globally generated for $m$ sufficiently divisible. But then so is its direct image under $g$, which is isomorphic to $\OO_X(m(D+A))$ by the projection formula since $mA$ is Cartier.

\end{proof}

Let us recall the following conjecture from \cite{Urb1}:

\begin{con} \label{ggweil}
Let $X$ be a projective normal variety. Then, for any divisor $D \in {\rm WDiv}_{\Q}(X)$, there exists a very ample divisor $A$ such that $\OO_X(mD)\otimes \OO_X(A)^{\otimes m}$ is globally generated for every $m\geq 1$. 
\end{con}

\begin{de} Let $X$ be a normal projective variety, $D$ any divisor on $X$ and $A$ an ample divisor. If there exists a $t \in \R$ such that $D+tA$ is quasi-nef, we define the quasi-nef threshold with respect to $A$ (qnt$_{A}$) as:
$$\text{qnt$_{A}$}(D)=\inf \{t\in \R \:| \: \OO_X(m(D + tA)) \mbox{ is globally generated for all $m$ sufficiently divisible}\}.$$ 
\end{de}

\begin{re} Let $X$ be a normal projective variety with at most log terminal singularities. For any divisor $D$ on $X$ and any ample divisor $A$, then qnt$_{A}(D)$ exists and it is a rational number. This is a direct consequence of the fact that for any variety with at most klt singularities, every divisorial ring is finitely generated \cite[Theorem 92]{KollarEx}.
\end{re}


\section{Non $\mathbb{Q}$-Gorenstein toric varieties}
The aim of this section is to give an explicit description fot toric varieties of the non-$\mathbb{Q}$-Gorenstein case. This is the more explicit example where it is possible to compute and understand the given definitions.

\subsection{Quasi-nef Divisors on Toric Varieties}

For the notation and basic properties of toric varieties we refer the reader to \cite{cox}. See \cite[Section 2.4]{BdFF} for a $b$-divisorial interpretation of the same problem.

Consider a normal projective toric variety $X=X_{\Sigma}$ corresponding to a complete fan $\Sigma$ in $N_{\R}$ (with no torus factor), with $\dim N_{\R}=n$. Recall that every $T_N$-invariant Weil divisor is represented by a sum $$D= \sum_{\rho \in \Sigma(1)} d_{\rho}D_{\rho},$$
where $\rho$ is a one-dimensional subcone (a ray), and $D_{\rho}$ is the associated $T_N$-invariant prime divisor. $D$ is Cartier if for every maximal dimension subcone $\sigma \in \Sigma(n)$, $D|_{U_{\sigma}}$ is locally a divisor of a character $\text{div}(\chi^{m_{\sigma}})$, with $m_{\sigma} \in N^{\vee}=M$. If $D$ is Cartier we will say that $\{m_{\sigma}|\sigma \in \Sigma (n)\}$ is the Cartier data of $D$. 

To every divisor we can associate a polyhedron:

$$P_D=\{m\in M_{\R} | \langle m, u_{\rho}\rangle\geq - d_{\rho} \text{ for every } \rho \in \Sigma(1) \}.$$
Even if the divisor is not Cartier, the polyhedron is still convex and rational but not necessarily integral.

For every divisor $D$ and every cone $\sigma \in \Sigma(n)$, we can describe the local sections as 
$$\OO_X(D)(U_{\sigma}) = \C[W]$$ 
where $W= \{\chi^m | \langle m, u_{\rho} \rangle + d_{\rho} \geq 0 \text{ for all } \rho \in \sigma(1)\}$.

Let us recall the following Proposition from \cite{lin}:

\begin{pro} For a torus invariant Weil divisor $D=\sum d_{\rho}D_{\rho}$, the following statements hold.
\begin{enumerate} 
\item $\Gamma(X, D)= \bigoplus_{m\in P_D\cap M} \C \cdot \chi^m$.
\item Given that $\OO_X(D)(U_{\sigma})= \C[\sigma^{\vee}\cap M]\langle\chi^{m_{\sigma, 1}}, \ldots, \chi^{m_{\sigma, r_{\sigma}}} \rangle$ is a finitely generated $\C[\sigma^{\vee}\cap M]$ module for every $\sigma \in \Sigma(n)$ and a minimal set of generators is assumed to be chosen, $\OO_X(D)$ is generated by its global sections if and only if $m_{\sigma ,j} \in P_D$ for all $\sigma$ and $j$.
\end{enumerate}
\end{pro}

We will also need the following result \cite{elizondo}.

\begin{teo} Let $X$ be a complete toric variety and let $D$ be a Cartier divisor on $X$. Then the ring 
$$\mathscr{R}_D:= \bigoplus_{n\geq 0} H^0(X, \OO_X(nD))$$
is a finitely generated $\C$-algebra.
\end{teo}

\begin{co} \label{fg} Since every toric variety admits a $\Q$-factorialization, a small morphism from a $\Q$-factorial variety (\cite[Corollary 3.6]{Fujino}), the above result holds for Weil divisors as well.
\end{co}

We then easily get the following result.

\begin{pro} \label{cholds} Conjecture \ref{ggweil} holds for $X=X_{\Sigma}$, a complete toric variety.
\end{pro}

We can now focus our attention on q-nef divisors.

\begin{re} A small birational map of toric varieties $f:Y \to X$ is given by adding faces of dimension $\geq2$ to the fan. This operation increases the number of subcones. In particular a subcone in the fan corresponding to $Y$ may be strictly contained in one of the original subcones.
\end{re}

\begin{pro}
Let us consider a Weil divisor $D \subseteq X$ on a  normal toric variety. If $f:Y \to X$ is a small birational map, then $P_D = P_{f^{-1}_*D}$.
\end{pro}

\begin{proof}  This is clear, since the definition of the polyhedron only depends on the rays generating the fan and not on the structure of the subcones.

\end{proof}

We assume that the polyhedron $P_D$ is of maximal dimension and that zero is inside the polyhedron.  

To have a better description of the relation between a small morphism and the local sections of a Weil divisor, we will introduce a new polyhedron associated to the divisor, the dual of $P_D$. 

\begin{de} Let $D= \sum d_{\rho}D_{\rho} \subseteq X=X_{\Sigma}$ be a Weil divisor on a  normal toric variety. We define $Q_D\subseteq N_{\R}$ to be the convex hull generated by $\frac{1}{d_{\rho}} u_{\rho}$ where $\rho \in \Sigma (1)$. In Particular $$Q_D=P_D^*= \{ u \in N_{\R} | \langle m, u \rangle \geq -1 \mbox{ for all } m \in P_D\}.$$
\end{de}

Recall that a divisor $D$ is Cartier if and only if for each $\sigma \in \Sigma$, there is $m_{\sigma} \in M$ with $\langle m_{\sigma}, u_{\rho}\rangle = -d_{\rho}$ for all $\rho \in \sigma(1)$, with $D|_{U_{\sigma}} = \text{div}(\chi^{-m_{\sigma}})$ (\cite[Theorem 4.2.8]{cox}). 

We will define $\Sigma'$ to be the fan generated by $\Sigma$ and the faces of $Q_D$. In particular, any face of $Q_D$ is contained in a hyperplane corresponding to $m_{\sigma}$ for some  $\sigma \in \Sigma(n)$. Note that the vertices of $Q_D$ are all contained in the $1$-dimensional faces of $\Sigma$, hence $Y:=X_{\Sigma'} \to X$ is a small birational map.

 For every cone $\sigma \in \Sigma (n)$, if $\OO_X(D)(U_{\sigma})$ is locally generated by a single equation (is locally Cartier) nothing changes. Otherwise we substitute the cone $\sigma$ by the set of subcones generated by the faces of $Q_D$ contained in $\sigma$.

\begin{lem} \label{toric-lemma} With the notation above, suppose that $\bar{\sigma} \subseteq \Sigma'(n)$ corresponds to a face of $Q_D$. Let $\bar{m} \in M_{\R}$ the element corresponding to the hyperplane containing the face, hence $\OO_X(D)(U_{\bar{\sigma}})= \C [\bar{\sigma}^{\vee} \cap M ] \langle \chi^{\bar{m}}\rangle$. If $\OO_X(D)$ is globally generated, then $\bar{m} \in P_D$.
\end{lem}

\begin{proof} Since $\chi^{\bar{m}}$ is a generator of $\OO_X(D)(U_{\bar{\sigma}})$, we have that $\langle \bar{m}, u_{\rho}\rangle = -d_{\rho}$ for every $\rho \in \bar{\sigma}(1)$. Also, since $Q_D$ is convex and $\langle \bar{m}, 0 \rangle =0 > -1$,  we have that $\langle \bar{m}, u_{\rho}\rangle \geq -d_{\rho}$ for every $\rho \in \Sigma(1)$, hence $\bar{m} \in P_D$.
\end{proof}

\begin{pro} \label{tor-qnef}
Let $X$ be a normal toric variety and $D$ a Weil divisor whose corresponding reflexive sheaf is generated by global sections. Then there exists a small map $f:Y \to X$ of toric varieties such that $\bar{D}:=f_*^{-1}D$ is $\Q$-Cartier and $f$-ample, and the vertices of $P_D$ are given by the Cartier data $\{m_{\sigma}|\sigma \in \Sigma' (n)\}$ of $\bar{D}$, where $\Sigma'$ is the fan associated to $Y$.
\end{pro}

\begin{proof}
We will consider the toric variety associated to the fan $\Sigma'$ generated by $\Sigma$ and the convex polytope $Q_D$. It follows from the construction that the divisor $\bar{D}$ is $\Q$-Cartier. By Lemma \ref{toric-lemma}, since $Q_D$ is convex, we have that the reflexive sheaf corresponding to $\bar{D}$ is still generated by global sections. 

Even more, every curve $C$ extracted via the map $f$ will correspond to a face $\tau \subseteq \Sigma'$. Since $\bar{D}$ is globally generated, we already know that $(\bar{D}.C)\geq 0$. In particular $\tau$ is given as the intersection of two maximal cones $\tau= \sigma \cap \sigma'$, and for each of the cones we have local generators of $\bar{D}$, $m$ and $m'$. The intersection is computed as $(\bar{D}.C)=\langle m,u\rangle- \langle m', u\rangle$, where $u$ is a ray in $\sigma\backslash \sigma'$, where this is zero if and only if $m=m'$, and this would not be one of the curves to be extracted by the map $f$ by definition.

\end{proof}

Because of Proposition \ref{tor-qnef} we get the following.

\begin{teo} Let $D$ be a Weil divisor on a normal toric variety $X$. Then D is q-nef if and only if $\OO_X(mD)$ is globally generated for $m \gg 0$.
\end{teo}

\begin{ex} Let $X$ be a normal projective toric variety and $A = \sum a_{\rho}D_{\rho}\subseteq X$ an ample divisor. Then $\text{qnt}_A(D)$ can be explicitly computed. In particular, let $D= \sum d_{\rho}D_{\rho}$ any Weil divisor. If no multiple of $D$ is  globally generated, this implies that for every $b \in \N$, there exists $\sigma_b \in \Sigma(n)$ such that $u_{\rho_b} \notin \sigma_b$ and $\langle m_{\sigma_b}^{bD}, u_{\rho_b} \rangle < -bd_{\rho_b}$, where $m_{\sigma_b}^{bD}$ is one of the generators of $\OO_X(bD)(U_{\sigma_b})$, i.e. there exists a positive rational number $\delta_{\rho_b}$ such that  $\langle m_{\sigma_b}^{bD}, u_{\rho_b} \rangle =  -bd_{\rho_b}  -\delta_{\rho_b}$. Since $A$ is ample, the support function of $A$ is strictly convex, and in particular $\langle m_{\sigma_a}^A, u_{\rho_b} \rangle > -a_{\rho_b}$, i.e. there exists a positive rational number $\varepsilon_{\rho_b}$ such that $\langle m_{\sigma_b}^A, u_{\rho_b} \rangle = -a_{\rho_b} + \varepsilon_{\rho_b}$. 

For every $\rho_b \notin \sigma_b$ and every $\sigma_b$ we can find a rational number $t_{b,\sigma_b, \rho_b}$ so that: 
$$\langle m_{\sigma_b}^D + t_{b,\sigma_b, \rho_b} m_{\sigma_b}^A, u_{\rho_b} \rangle =  -bd_{\rho_b} -\delta_{\rho_b} -t_{b,\sigma_b, \rho_b} a_{\rho_b} + t_{b,\sigma_b, \rho_b} \varepsilon_{\rho_b} = -bd_{\rho_b} -t_{b, \sigma_b,\rho_b} a_{\rho_b}.$$

Then $$\text{qnt}_A(D)= \inf_b{\max_{\sigma_b, \rho_b\notin \sigma_b} -t_{b, \sigma_b, \rho_b}}.$$

\end{ex}

\subsection{Minimal Log Discrepancies for Terminal Toric Threefolds}
In this last section we go back to properties of log discrepancies in the setting of \cite{dFH} in the context of toric varieties.

Depending on our choice of a relative canonical divisor, we have two possible definitions for the Minimal Log Discrepancies (MLD's). 

\begin{de}
Let $X$ be a normal variety over the complex numbers, we associate two numbers to the variety $X$:

$$\text{MLD$^-$}(X)= \inf_{E} \text{ord}_E(K^-_{Y/X})$$

and

$$\text{MLD$^+$}(X)= \inf_{E} \text{ord}_E(K^+_{Y/X})$$

where $E \subseteq Y$ is any prime divisor and $Y \to X$ is any proper birational morphism of normal varieties.
\end{de}

It is natural to wonder if these MLD's also satisfy the ACC conjecture. If $X$ is assumed to be $\Q$-Gorenstein, then this is conjectured to hold by V. Shokurov. In view of \cite[Theorem 5.4]{dFH}, the MLD$^+$'s correspond to MLD's of appropriate pairs $(X, \Delta)$. However the coefficients of $\Delta$ do not necessarily belong to a DCC set  (the investigation for log-pairs given in the $\mathbb{Q}$-Gorenstein context has this extra assumption in \cite{Ambro}). 

\begin{pro}
The set of all MLD$^+$'s for terminal toric threefolds does not satisfy the ACC conjecture.
\end{pro}

\begin{proof}
We give an explicit example of a set of terminal toric threefolds whose associated MLD$^+$'s converge to a number from below. The problem is local, hence we will  consider a set of affine toric threefolds given by the following data.

Let $X$ be the affine toric variety associated to the cone $\sigma= \langle u_1, u_2, u_3, u_4 \rangle$, $u_1=(2, -1, 0)$, $u_2=(2, 0, 1)$, $u_3=(1,1,1)$, $u_4=(a, 1, 0)$  with $a\in \N$. The associated toric variety in non-$\Q$-Gorenstein, i.e. the canonical divisor $K_X= \sum -D_i$ is not $\Q$-Cartier.

Let $\Delta=\sum d_iD_i$ be a $\Q$-divisor such that $0\leq d_i \leq 1$ and $-K_X + \Delta$ is $\Q$-Cartier. This means that there exists $m= (x, y, z)$ such that $-K_X+ \Delta = \sum (m, u_i)D_i$. Hence $\Delta = (2x -y -1 )D_1 + (2x+z -1)D_2 +(x+y+z-1)D_3 + (ax+y-1)D_4$.

The exceptional divisor $E$ giving the smallest discrepancy is the one corresponding to the element $u_E=u_1 +u_2 +u_3$ (it is the exceptional divisor generated by the ray of smallest norm). In particular, we have 
$$\text{ord}_E(K^+_{Y/X})= \inf_{\Delta \text{ boundary}} 5x -2z.$$ 
Increasing the value of the parameter $a$ we see that the minimal valuation is given  by $\displaystyle{\frac{4a+5}{a+2}}$ which accumulates from below at the value $4$. Notice that the key is  solving a problem of minimality with constrains, corresponding to the fact that the coefficients of $\Delta$ need to be in the interval $[0,1)$.
\end{proof}

\subsection{Towards a revised toric MMP}

Let $X=X_{\Sigma}$ as in the previous section. With using Theorem \ref{cholds}, we will try to define a Minimal Model Program for any normal projective toric variety without attaching any boundary to it.

Let us recall the following well known fact \cite[Theorem 3.25 (3)]{KM}:

\begin{pro}
Let $(Y, \Delta)$ be a klt pair, $\Delta$ effective, and $f: Y\to X$ a projective morphism of normal projective varieties. Let $F \subset \overline{NE}(Y/X)$ be a $(K_Y+\Delta)$-negative extremal face. Then there is a unique morphism $cont_F: Y/X \to Z/X$ 
\end{pro}

Using this proposition and the MMP with scaling we will prove the following:

\begin{pro} 
For any normal projective toric variety $X$, there exists a small birational morphism $f: Z \to X$ such that $K_Z$ is nef over $X$.
\end{pro}

\begin{proof} Let $H\subset X$ be an ample divisor such that $K_X+H$ is q-nef (we know that this is possible because of Theorem 
We know that $f$ is a small birational morphism and that $K_Y+ f^*H$ is nef over $X$ (\cite[Proposition 6.2]{KM}), note that $Y$ is $\Q$-Gorenstein. At this point we run a Minimal Model with scaling of $H$ over $X$ to obtain another small morphism $\pi: Z \to X$ 
\end{proof}

\begin{pro}
Let $\pi: Z \to X$ be a small morphism of toric varieties such that $K_Z$ is nef over $X$. Then for any $K_Z$-negative curve $C$, there exists a boundary $\Delta$ on $X$ such that $(K_X +\Delta).\pi_*C <0$ i.e. every contraction on $Z$ arising from the cone theorem, induces a contraction on $X$. 
\end{pro}

\begin{proof}
Let $C \subset Z$ such that $K_Z.C<0$. Let $D$ be an ample effective Cartier divisor on $X$, in particular $D.C >0$. Also, writing $D= \sum d_{\rho}D_{\rho}$, we can assume that there exists an integer $k>0$ such that $D' :=D/k$ is a $\Q$-Cartier divisor such that $|d_{\rho}| <1$ for every $\rho \in \Sigma(1)$. Let us choose $\Delta = -K_X -D'$, that by assumption is an effective divisor and $K_X+ \Delta = -D'$ is $\Q$-Cartier and $(K_X+ \Delta).C <0$.

%
\end{proof}

\addcontentsline{toc}{chapter}{Bibliography}
\nocite{*}

\bibliographystyle{alpha}  
\bibliography{bibliografia}

  \vskip .2cm
   Stefano Urbinati,
   Universit\`a degli Studi di Padova,
   Dipartimento di Matematica,
   ROOM 630, Via Trieste 63,
   35121 Padova, Italy.

   \nopagebreak
\noindent   \textit{E-mail address:} \texttt{urbinati.st@gmail.com}

\end{document}